\documentclass[oneside]{amsart}
\usepackage[utf8]{inputenc}
\usepackage[percent]{overpic}
\usepackage{amsmath}
\usepackage{amssymb}
\usepackage{tikz}
\usepackage{tikz-cd}
\usepackage{mathtools}
\usepackage{stmaryrd}
\usepackage{amsthm}
\usepackage{array} 
\usepackage[english]{babel}
\usepackage{enumitem}
\usepackage{lmodern}
\usepackage{xfrac}  

\usepackage[paper=a4paper,  margin=2.991cm]{geometry}

\usepackage{csquotes}
\usepackage[style=alphabetic, backend=biber, maxbibnames=99, maxalphanames=4, minalphanames=4]{biblatex}
\DeclareFieldFormat{titlecase}{\ifcurrentfield{title}{\MakeSentenceCase*{#1}}{#1}}
\usepackage[dvipsnames]{xcolor}
\usepackage[hidelinks]{hyperref}
\hypersetup{
    pdftitle={Projected Brion's Theorem},
    pdfauthor={Dupraz, M.; Monin, L.},
    pdfkeywords= {Brion's theorem; integer-point transform; lattice projection; contracted class; equivariant localization; linear family of polytopes; Gelfand-Zetlin polytope; string polytope},
    pdfsubject= {},
    colorlinks=true, 
    citecolor=PineGreen, 
    linkcolor=RubineRed, 
    urlcolor=PineGreen
}

\usepackage{cleveref}

\usepackage{aliascnt}
\newcommand{\sharedtheorem}[2]{%
  \newaliascnt{#1}{theorem}\newtheorem{#1}[#1]{#2}\aliascntresetthe{#1}}

\theoremstyle{plain}
\newtheorem{theorem}{Theorem}[section]
\sharedtheorem{lemma}{Lemma}
\sharedtheorem{proposition}{Proposition}
\sharedtheorem{corollary}{Corollary}

\theoremstyle{definition}
\sharedtheorem{definition}{Definition}
\sharedtheorem{defprop}{Definition--Proposition}
\sharedtheorem{remark}{Remark}
\sharedtheorem{notation}{Notation}
\sharedtheorem{example}{Example}
\sharedtheorem{question}{Question}
\crefname{lemma}{Lemma}{Lemmas}
\crefname{proposition}{Proposition}{Propositions}
\crefname{corollary}{Corollary}{Corollaries}
\crefname{definition}{Definition}{Definitions}
\crefname{defprop}{Definition--Proposition}{Definition--Propositions}
\crefname{remark}{Remark}{Remarks}
\crefname{notation}{Notation}{Notations}
\crefname{example}{Example}{Examples}

\crefname{question}{question}{questions}
\Crefname{question}{Question}{Questions}

\newcommand{\rleft}{\mathopen{}\mathclose\bgroup\left}
\newcommand{\rright}{\aftergroup\egroup\right}

\newcommand{\C}{{\mathbb{C}}}
\newcommand{\K}{{\mathbb{K}}}
\newcommand{\R}{{\mathbb{R}}}
\newcommand{\Z}{{\mathbb{Z}}}
\newcommand{\N}{{\mathbb{N}}}

\newcommand{\Q}{{\mathbb{Q}}}

\renewcommand{\vec}[1]{\ensuremath{\mathbf{#1}}}

\newcommand{\cF}{{\mathcal{F}}}

\newcommand{\cP}{{\mathcal{P}}}

\newcommand{\cL}{{\mathcal{L}}}

\newcommand{\cV}{{\mathcal{V}}}
\newcommand{\cS}{{\mathcal{S}}}

\newcommand{\Conv}{\operatorname{Conv}}
\newcommand{\Cone}{\operatorname{Cone}}
\newcommand{\ord}{\operatorname{ord}}
\newcommand{\rec}{\operatorname{rec}}

\newcommand{\supp}{\operatorname{supp}}
\newcommand{\relint}{\operatorname{relint}}
\newcommand{\gen}{\operatorname{gen}}

\usepackage{pgf,tikz}
\usetikzlibrary{arrows,patterns}

\pgfdeclarepatternformonly{north east lines wide}%
  {\pgfqpoint{-1pt}{-1pt}}{\pgfqpoint{6pt}{6pt}}{\pgfqpoint{5pt}{5pt}}%
  {\pgfsetlinewidth{0.4pt}%
   \pgfpathmoveto{\pgfqpoint{0pt}{0pt}}%
   \pgfpathlineto{\pgfqpoint{6.1pt}{6.1pt}}%
   \pgfusepath{stroke}}
\pgfdeclarepatternformonly{north west lines wide}%
  {\pgfqpoint{-1pt}{-1pt}}{\pgfqpoint{6pt}{6pt}}{\pgfqpoint{5pt}{5pt}}%
  {\pgfsetlinewidth{0.4pt}%
   \pgfpathmoveto{\pgfqpoint{0pt}{6pt}}%
   \pgfpathlineto{\pgfqpoint{6.1pt}{-0.1pt}}%
   \pgfusepath{stroke}}

\title{Projected Brion's Theorem}
\author{Matthew Dupraz
    \and
    Leonid Monin
}

\newcommand{\Addresses}{{
  \bigskip
  \footnotesize

  Matthew~Dupraz, \textsc{Institut für Mathematik, Freie Universität Berlin}\par\nopagebreak
  \textit{E-mail address}: \texttt{matthew.dupraz@fu-berlin.de}

  \medskip

  Leonid~Monin, \textsc{\'Ecole Polytechnique F\'ed\'erale de Lausanne}\par\nopagebreak
  \textit{E-mail address}: \texttt{leonid.monin@epfl.ch}

}}


\begin{document}

\begin{abstract}
    We present an analogue of Brion's theorem for the projected integer-point transform.
    Brion's theorem writes the integer-point transform of a rational polyhedron $P$ as the
    sum of the rational functions associated to its vertex cones.
    The pushforward $\pi_*$ along a lattice projection $\pi$ turns the transform of $P$
    into the generating function of the numbers of lattice points in the fibres of $\pi$.
    We group the vertices of $P$ into \emph{contracted classes},
    generated by the edges whose direction $\pi$ kills, and show that the sum of the rational functions associated to the vertex cones in a contracted class can be pushed forward along $\pi$, yielding a rational function decomposition of the projected integer-point transform.
    For a lattice polytope, the classes correspond to the connected components of the fixed locus of the associated toric variety under the action of a subtorus, and our formula resembles the
    localization formula for this subtorus.
    We use this decomposition to provide an explicit formula for the projected integer-point transform of a linear family of polytopes. We illustrate this result with families of Gelfand--Zetlin and string polytopes.
\end{abstract}

\maketitle

\section{Introduction}
\label{sec:intro}

Let $P \subseteq \R^n$ be a rational polyhedron and let
\[
    \sigma(P; \vec x) = \sum_{\vec a \in P \cap \Z^n} \vec x^{\vec a}
\]
be its integer-point transform, which is a formal Laurent series encoding the lattice points of $P$.
It is well known that  $ \sigma(P; \vec x)$ corresponds to a rational function, which we denote by
$\varsigma(P; \vec x)$ (\Cref{def:varsigma}). Brion's theorem (\cite{Brion88} for lattice polytopes,
\cite[Theorem 3.5]{Barvinok1999AnAT} for rational polyhedra) states that
\[
    \varsigma(P; \vec x) = \sum_{v \textrm{ vertex of } P} \varsigma(K_v; \vec x),
\]
where $K_v$ is the vertex cone of $P$ at $v$.

Let $\pi: \R^n \to \R^m$ be a projection with
$\pi(\Z^n) \subseteq \Z^m$. When the fibres of $\pi$ meet $P$ in finitely many lattice points, the
induced map $\pi_*: \vec x^{\vec a} \mapsto \vec y^{\pi(\vec a)}$ carries $\sigma(P; \vec x)$ to
the generating function of the fibre counts $\#(P \cap \pi^{-1}(\vec b) \cap \Z^n)$
(\Cref{fig:intro-projection}).

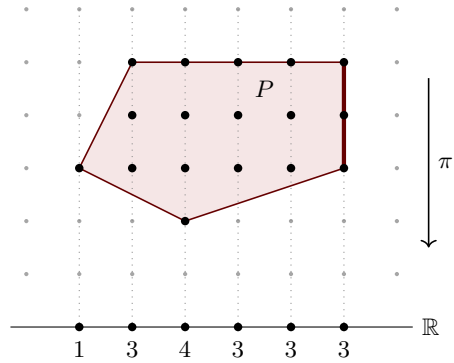
\begin{figure}[ht]
    \centering
    \begin{tikzpicture}[scale=0.7, font=\small, line join=round]
        \foreach \i in {-1, ..., 6} \foreach \j in {-1, ..., 4} \fill[black!35] (\i, \j) circle (0.04);
        \fill[red!60!black!10] (0, 1) -- (2, 0) -- (5, 1) -- (5, 3) -- (1, 3) -- cycle;
        \foreach \i in {0, ..., 5} \draw[thin, dotted, black!45] (\i, 4) -- (\i, -2);
        \draw[semithick, red!40!black] (0, 1) -- (2, 0) -- (5, 1) -- (5, 3) -- (1, 3) -- cycle;
        \draw[line width=1.6pt, red!40!black] (5, 1) -- (5, 3);
        \foreach \p in {(0, 1), (1, 1), (1, 2), (1, 3), (2, 0), (2, 1), (2, 2), (2, 3), (3, 1),
            (3, 2), (3, 3), (4, 1), (4, 2), (4, 3), (5, 1), (5, 2), (5, 3)} \fill \p circle (0.08);
        \node[fill=red!60!black!10, inner sep=1pt] at (3.5, 2.5) {$P$};
        \draw[->, semithick] (6.6, 2.7) -- node[right] {$\pi$} (6.6, -0.5);
        \draw[thin] (-1.3, -2) -- (6.3, -2);
        \foreach \i/\n in {0/1, 1/3, 2/4, 3/3, 4/3, 5/3} {
            \fill (\i, -2) circle (0.08);
            \node[below=2pt] at (\i, -2) {$\n$};
        }
        \node[right] at (6.3, -2) {$\R$};
    \end{tikzpicture}
    \caption{A lattice polygon $P \subseteq \R^2$ and the projection $\pi(a_1, a_2) = a_1$. The number
    below each point of $\Z$ is the number of lattice points of $P$ in its fibre, so that
    $\pi_*\sigma(P; \vec x) = 1 + 3y + 4y^2 + 3y^3 + 3y^4 + 3y^5$. The thick edge is contracted,
    and its two vertices form a contracted class.}
    \label{fig:intro-projection}
\end{figure}

One would like a formula for $\pi_*\sigma(P; \vec x)$
assembled from the local data at the vertices of $P$, as Brion's theorem supplies before the projection.
Applying $\pi_*$ to each summand coming from Brion's theorem does not work.
Indeed, $\pi_*$ is defined only for rational functions that can be written with a denominator outside of
$\ker \pi_*$. The rational function associated to a
vertex cone has a denominator with factors $1 - \vec x^{\vec u}$, one
for each ray with primitive direction $\vec u$, and $1 - \vec x^{\vec u}$ lies in $\ker \pi_*$ exactly when $\pi(\vec u) = 0$.
We resolve this issue by grouping the vertices into \emph{contracted classes}.

Call an edge of $P$ \emph{contracted} if $\pi$ kills its direction. For a bounded edge, this means
that its endpoints have the same image under $\pi$. The contracted edges generate an equivalence relation on the set of vertices of $P$, whose equivalence classes we call the
\emph{contracted classes} of $P$. Our main result (\Cref{thm:projected-brion}) states that if
no unbounded edge of $P$ is contracted, the rational functions of the vertex
cones in a contracted class add up to a rational function with a well-defined pushforward, and
\[
    \pi_*\varsigma(P; \vec x) = \sum_{S} \pi_*\left(\sum_{v \in S} \varsigma(K_v; \vec x)\right),
\]
the outer sum running over the contracted classes. Furthermore, for lattice polyhedra, \Cref{prop:minimality} shows that contracted classes yield the finest grouping that makes this work.

To relate the left-hand side of the equation to the projected integer-point transform $\pi_*\sigma(P; \vec x)$, we rely on the fact that the pushforward $\pi_*$ is well behaved with respect to taking rational functions of Laurent series. To any \emph{summable} Laurent series $f$, we may assign a rational function $\Phi(f)$ (\Cref{def:summable}), and one should expect that taking $\Phi$ commutes with pushing forward along the projection. In \Cref{thm:phi-localizes} we show that when a summable $f$ has support with finite fibres with respect to the projection $\pi$, the pushforward $\pi_*(f)$ is again summable. Moreover, $\Phi(f)$ can be written with a denominator avoiding $\ker \pi_*$, so it can be pushed forward, and
\[\Phi(\pi_*(f)) = \pi_*(\Phi(f)).\]
Applied to $\sigma(P; \vec x)$, this shows that if
$P$ meets every fibre of $\pi$ in finitely many lattice points, then $\pi_*\varsigma(P; \vec x)$ is
the rational function associated to $\pi_*\sigma(P; \vec x)$.

\begin{example}[Toric varieties]\label{ex:intro-toric}
    Let $P$ be a full-dimensional lattice polytope, $X$ its toric variety over $\C$ and $\cL$ the ample line bundle
    corresponding to $P$. Then $\sigma(P; \vec x)$ is the character of $H^0(X, \cL)$, and $\pi_*$
    restricts it to the subtorus $T'$ dual to $\pi$. The contracted classes correspond to the
    connected components of the fixed locus $X^{T'}$, and \Cref{thm:projected-brion} is then
    Thomason's localization theorem in $T$-equivariant $K$-theory, for the subtorus $T'$ (\Cref{rem:localization}).
\end{example}

\begin{example}[Weight multiplicities]\label{ex:intro-weights}
    Let $P$ be the Gelfand--Zetlin polytope of a regular dominant weight $\lambda$ and $\pi$ its weight
    map. The fibre counts are the weight multiplicities of the irreducible representation of
    highest weight $\lambda$. The weight map contracts no edge of $P$ (\Cref{sec:gz-string}), so
    every contracted class is a single vertex and \Cref{thm:projected-brion} is Brion's theorem pushed
    forward term by term. In fact, \cite[Theorem 1.3]{Makhlin16} shows that this  recovers Weyl's character formula. On the other hand, for string
    polytopes, which generalize Gelfand--Zetlin polytopes to other groups and reduced words, the weight
    map can contract edges, so that the contracted classes need not be singletons
    (\Cref{ex:string-polytopes}).
\end{example}

\Cref{thm:projected-brion} extends to twisted pushforwards
$\vec x^{\vec a} \mapsto \theta(\vec a)\, \vec y^{\pi(\vec a)}$ for a homomorphism
$\theta: \Z^n \to \K^\times$, where we say that an edge is contracted 
if $\theta$ is also trivial on
its direction. The theorem can also be used to study
transforms with values in the group ring of a finitely generated abelian
group $G$, which count the lattice points in the fibres of a homomorphism $\Z^n \to G$
(\Cref{rem:twisted}).

For a linear family of lattice polytopes $\Delta(\vec c)$, \Cref{thm:linear-family} shows that the terms of
\Cref{thm:projected-brion} depend in a precise way on $\vec c$, i.e.,
\[
    \pi_*\varsigma(\Delta(\vec c); \vec x) = \sum_S \vec y^{\vec a_S(\vec c)}\,
    \frac{g_S(\vec c)}{\prod_{\vec u}\bigl(1 - \vec y^{\pi(\vec u)}\bigr)},
\]
where $\vec u$ runs over the edge directions at $S$ not contracted by $\pi$, $\vec a_S$ is linear,
and $g_S \in \K[\vec y^{\pm 1}][\vec c]$ is a polynomial in $\vec c$ of degree bounded by the 
dimension of the affine span of the vertices in $S$. For string polytopes, this provides one way to
give Littelmann's character formula the shape needed to extend it to all weights
(\Cref{sec:gz-string}).

\subsection*{Related work}

Our setting is very close to those of Berline--Vergne and of Peterson, who study the sum
\[
    S(P; \xi) = \sum_{\vec m \in \Z^n \cap P} e^{\langle \xi, \vec m \rangle}
\]
for $\xi$ varying over complex-valued linear functionals.
Up to a change of variables, this is the integer-point transform, and Brion's theorem also applies, so that
\[
    S(P; \xi) = \sum_{v\text{ vertex of }P}S(K_v; \xi),
\]
where each $S(K_v; \xi)$ is a meromorphic function of $\xi$.
Berline and Vergne show in \cite[Theorem 58]{BerlineVergne12} that for simple
polytopes, grouping the vertices of faces on which some $\xi$ is constant gives a function which is holomorphic at $\xi$.
This agrees with \Cref{thm:projected-brion}, as for simple polytopes,
the contracted classes are vertex sets of faces, but in general they need
not be (\Cref{ex:pyramid-classes}).

Peterson writes an analogous degenerate Brion formula in \cite{Peterson25}.
Rather than grouping Brion's vertex terms, he writes $P$ as a signed sum of indicator functions of polyhedra depending on a fixed $\xi$, whose transforms are all holomorphic at $\xi$.

When no edge is contracted, Brion and Vergne push a parametric Brion formula forward term by term
\cite[Section 3.4]{BrionVergne97}, and \Cref{thm:linear-family} extends this to arbitrary $\pi$.
For linear families we compare \Cref{thm:linear-family} with earlier work of
Peterson \cite{Peterson23} in \Cref{rem:peterson}.

\subsection*{Acknowledgements}
The authors would like to thank Christian Haase,
Igor Makhlin, Evgeny Smirnov for insightful discussions on the topic. The main part of the work was done while LM was a visitor at the ETH
Institute for Theoretical Studies. LM thanks the ITS for its hospitality and inspiring research environment.

The work of MD and LM was supported by the SPP 2458 ``Combinatorial Synergies'', funded by the
Deutsche Forschungsgemeinschaft (DFG, German Research Foundation) -- project number 539974215, as
well as by the Swiss National Science Foundation (SNSF) -- grant 224099.

\subsection*{Declaration of AI use}

The authors used the AI model Claude Opus 5.5 (Anthropic) to edit the exposition,
proofread, create figures, and carry out computer experiments with string polytopes.
The proof of \Cref{thm:phi-localizes} was developed with the help of Claude. The authors verified all
mathematical content and take full responsibility for it.

\section{Summable series and their pushforward}\label{sec:transform}

\subsection{Summable series}\label{subsec:summable}

Fix an arbitrary field $\K$, and write $\K[\vec x^{\pm 1}] := \K[x_1^{\pm 1}, \dots, x_n^{\pm 1}]$ for the ring
of Laurent polynomials, and $\K\llbracket \vec x^{\pm 1}\rrbracket$ for the space of formal Laurent
series, which are arbitrary formal sums $\sum_{\vec a \in \Z^n} c_{\vec a} \vec x^{\vec a}$. Generally, it is not possible to multiply two such series, because infinitely many terms may contribute to a given monomial. On the other hand, multiplication by
a Laurent polynomial is well defined and makes $\K\llbracket \vec x^{\pm 1}\rrbracket$ a module over
$\K[\vec x^{\pm 1}]$. The fraction field of $\K[\vec x^{\pm 1}]$ is the rational function field
$\K(\vec x)$.

\begin{defprop}[{\cite[Definition 3.6]{HaaseNillPaffenholz21}}]\label{def:summable}
    We define the subspace $\cS_{\vec x} \subseteq \K\llbracket \vec x^{\pm 1} \rrbracket$ of
    \emph{summable Laurent series} to be the set of Laurent series
    $f \in \K\llbracket \vec x^{\pm 1} \rrbracket$ for which there exists a nonzero Laurent
    polynomial $q \in \K[\vec x^{\pm 1}]$ such that $q \cdot f$ is a Laurent polynomial $p$. We then define
    \[
        \Phi: \cS_{\vec x} \to \K(\vec x), \qquad f \mapsto p/q.
    \]
    Then $\Phi$ is well defined, a $\K[\vec x^{\pm 1}]$-module homomorphism, and is moreover multiplicative: if $f_1, f_2 \in \cS_{\vec x}$ and every coefficient of
    $f_1 f_2$ is a finite sum, then $f_1 f_2 \in \cS_{\vec x}$ and $\Phi(f_1 f_2) = \Phi(f_1)\Phi(f_2)$.
\end{defprop}

\begin{proof}
    If $q f = p$ and $q' f = p'$, then $q' p = q' q f = q p'$, so that $p/q = p'/q'$, which shows that $\Phi$ is well defined. If
    $q_1 f_1 = p_1$ and $q_2 f_2 = p_2$, then $q_1 q_2 (f_1 + f_2) = q_2 p_1 + q_1 p_2$, and
    $q_1 (g f_1) = g p_1$ for $g \in \K[\vec x^{\pm 1}]$, so that $\cS_{\vec x}$ is a
    $\K[\vec x^{\pm 1}]$-submodule and $\Phi$ is a module homomorphism. Finally,
    $q_1 q_2 \cdot f_1 f_2 = (q_1 f_1)(q_2 f_2) = p_1 p_2$, since the coefficient of a monomial on
    either side is the same sum, which has only finitely many nonzero terms because $q_1$ and $q_2$
    are Laurent polynomials and the coefficients of $f_1 f_2$ are finite sums.
\end{proof}

Let $\pi: \R^n \to \R^m$ be a projection, with
$\pi(\Z^n) \subseteq \Z^m$, and let $\vec y := (y_1, \dots, y_m)$.
The projection $\pi$
induces a pushforward map
\begin{align*}
    \pi_* : \K[\vec x^{\pm 1}] &\to \K[\vec y^{\pm 1}]\\
    \vec x^{\vec a} &\mapsto \vec y^{\pi(\vec a)}.
\end{align*}
Let
$f = \sum_{\vec a \in \Z^n} f_{\vec a} \vec x^{\vec a} \in \K\llbracket \vec x^{\pm 1} \rrbracket$
be a Laurent series and denote by $\supp(f)$ the support of $f$, i.e., the set
$\{\vec a \in \Z^n : f_{\vec a} \neq 0\}$. Let
\[
    \cF_{\pi} := \{f \in \K\llbracket \vec x^{\pm 1} \rrbracket : \textrm{for all } \vec b \in \Z^m, \#(\pi^{-1}(\vec b) \cap \supp(f)) < \infty \},
\]
the set of Laurent series whose support meets every fibre of $\pi$ in finitely many points.
For any $f \in \cF_{\pi}$, we have a well-defined pushforward
\[
    \pi_*(f) = \sum_{\vec a \in \Z^n}f_{\vec a}\vec y^{\pi(\vec a)}
    = \sum_{\vec b \in \Z^m}\left(\sum_{\vec a \in \pi^{-1}(\vec b) \cap \Z^n}f_{\vec a}
    \right)\vec y^{\vec b}.
\]
The space $\cF_\pi$ is closed under multiplication by Laurent polynomials, and for all $q \in \K[\vec x^{\pm 1}]$, we have that
$\pi_*(q \cdot f) = \pi_*(q)\,\pi_*(f)$.

The kernel of the map $\pi_*$ is prime,
as its image is an integral domain; we may therefore localize at $\ker \pi_*$ and extend $\pi_*$ to
a homomorphism
\[
    \pi_*: \K[\vec x^{\pm 1}]_{\ker \pi_*} \to \K(\vec y), \qquad p/q \mapsto \pi_*(p)/\pi_*(q),
\]
where $\K[\vec x^{\pm 1}]_{\ker \pi_*} \subseteq \K(\vec x)$ consists of those rational functions
that can be written as a fraction whose denominator avoids $\ker \pi_*$.
We thus have two notions of
pushforward, one for Laurent series and one for rational functions, and we would like to compare them.

We denote by $\cS_{\vec y}$ the set of summable Laurent series in the 
variables $\vec y$, and the homomorphism taking a summable Laurent series to its rational function also by $\Phi$.
Let
\[
    \cS_{\pi} := \cS_{\vec x} \cap \cF_{\pi}.
\]
This is the largest set of Laurent series $f$ on which both $\Phi$ and $\pi_*$ are well defined, and on
it the two notions of pushforward agree, as the following theorem shows.

\begin{theorem}\label{thm:phi-localizes}
    Let $f \in \cS_{\pi}$. Then $\pi_*(f)$ is summable, $\Phi(f)$ lies in
    $\K[\vec x^{\pm 1}]_{\ker \pi_*}$, and
    \[
        \Phi\left(\pi_*(f)\right) = \pi_*\left(\Phi(f)\right).
    \]
\end{theorem}

\begin{proof}
    \emph{Setup and expansion in $\vec w$.} We may assume $\pi(\Z^n) = \Z^m$, since replacing $\Z^m$ by
    $\pi(\Z^n)$ does not change $\ker \pi_*$, and as $\K[\pi(\Z^n)] \to \K[\vec y^{\pm 1}]$ is
    injective, it changes neither whether $\pi_*(f)$ is summable nor its rational function. As
    $\Z^m$ is free, $\Z^n \cong (\ker \pi \cap \Z^n) \oplus \Z^m$, which writes
    $\K[\vec x^{\pm 1}] = \K[\vec y^{\pm 1}][\vec z^{\pm 1}]$, where $\vec z = (z_1, \dots, z_d)$
    are the variables of $\ker \pi \cap \Z^n$ and $d := \dim \ker \pi$. Furthermore, we have that
    $\cF_\pi = \K[\vec z^{\pm 1}]\llbracket \vec y^{\pm 1}\rrbracket$. Put $w_i := z_i - 1$. Since
    $z_i = 1 + w_i$ is a unit of $\K\llbracket \vec w \rrbracket$, we have
    $\K[\vec z^{\pm 1}] \subseteq \K\llbracket \vec w \rrbracket$ and so we may see $\cF_\pi$ as a
    subset of $\K\llbracket \vec w \rrbracket\llbracket \vec y^{\pm 1} \rrbracket
    = \K\llbracket \vec y^{\pm 1} \rrbracket\llbracket \vec w \rrbracket$. Under this
    identification $\pi_*: \cF_\pi \to \K\llbracket \vec y^{\pm 1} \rrbracket$ sets $\vec w = 0$.
    Fix a denominator $q \in \K[\vec x^{\pm 1}]$ of $f$. Multiplying it by $\vec z^{\vec c}$ for a
    suitable $\vec c \in \Z^d$, we may assume that $q$ and $p := q \cdot f$ lie in
    $\K[\vec y^{\pm 1}][\vec z] = \K[\vec y^{\pm 1}][\vec w]$.
    Grouping terms by their power of $\vec w$, we write
    \[
        q = \sum_{\beta \in \N^d} q_\beta \vec w^\beta, \qquad p = \sum_{\beta \in \N^d} p_\beta \vec w^\beta, \qquad f = \sum_{\beta \in \N^d} f_\beta \vec w^\beta
    \]
    with $q_\beta, p_\beta \in \K[\vec y^{\pm 1}]$ and
    $f_\beta \in \K\llbracket \vec y^{\pm 1} \rrbracket$. As $\pi_*$ sets $\vec w = 0$, we have
    $f_0 = \pi_*(f)$. Comparing the coefficients of $\vec w^\gamma$ in $q \cdot f = p$ gives, in
    $\K\llbracket \vec y^{\pm 1} \rrbracket$,
    \begin{equation}\label{eq:graded-denominator}
        \sum_{\alpha \leq \gamma} q_\alpha f_{\gamma - \alpha} = p_\gamma \qquad \text{for all } \gamma \in \N^d,
    \end{equation}
    where $\alpha \leq \gamma$ is meant componentwise.

    \emph{Summability of the $f_\beta$.} We show that every $f_\beta$ lies in $\cS_{\vec y}$. Fix a
    monomial order $\prec$ on $\N^d$, and let $\varepsilon$ be the least $\alpha$ with
    $q_\alpha \neq 0$. For $\gamma = \varepsilon + \beta$, \eqref{eq:graded-denominator} reads
    \[
        q_\varepsilon f_\beta = p_{\varepsilon + \beta} - \sum_{\substack{\alpha \succ \varepsilon \\ \alpha \leq \varepsilon + \beta}} q_\alpha f_{\varepsilon + \beta - \alpha},
    \]
    where each term of the sum has $\varepsilon + \beta - \alpha \prec \beta$. For $\beta = 0$ the
    sum is empty, as $\alpha \leq \varepsilon$ implies $\alpha \preceq \varepsilon$, so that
    $q_\varepsilon f_0 = p_\varepsilon$, and $f_0 = \pi_*(f)$ is summable. By induction along
    $\prec$ the right-hand side lies in $\cS_{\vec y}$, and hence so does $f_\beta$. Indeed, if $r$ is a
    denominator of $q_\varepsilon f_\beta$, then $r q_\varepsilon$ is one of $f_\beta$, as
    $q_\varepsilon \neq 0$.

    \emph{Conclusion.} The ideal $(\vec w)$ generated by $w_1, \dots, w_d$ is $\ker \pi_*$, so
    $A := \K[\vec x^{\pm 1}]_{\ker \pi_*}$ is a Noetherian local ring with maximal ideal $(\vec w)$.
    It contains $\K(\vec y)$, as the nonzero elements of $\K[\vec y^{\pm 1}]$ lie outside
    $(\vec w)$. Let $\varphi_\beta := \Phi(f_\beta) \in \K(\vec y)$
    and $\widehat \varphi := \sum_\beta \varphi_\beta \vec w^\beta \in \K(\vec y)\llbracket \vec w \rrbracket$.
    Applying $\Phi: \cS_{\vec y} \to \K(\vec y)$, which is $\K[\vec y^{\pm 1}]$-linear and the
    identity on $\K[\vec y^{\pm 1}]$, to \eqref{eq:graded-denominator} gives
    $q \cdot \widehat \varphi = p$. For $k \in \N$, the truncation
    $\varphi_{<k} := \sum_{|\beta| < k} \varphi_\beta \vec w^\beta$ lies in $A$, and
    $p - q\varphi_{<k} = q(\widehat \varphi - \varphi_{<k})$ is a polynomial in $\vec w$ without
    terms of degree less than $k$, so that $p \in qA + (\vec w)^k$. By Krull's intersection theorem
    \cite[Theorem 8.10(i)]{Matsumura86}, $\bigcap_k (qA + (\vec w)^k) = qA$, so $p \in qA$
    and hence $\Phi(f) = p/q \in A$. Note that $A$ is a subring of
    $\K(\vec y)\llbracket \vec w \rrbracket$, and as $q \cdot \Phi(f) = p = q \cdot \widehat \varphi$
    there, $\Phi(f) = \widehat \varphi$. The map $\pi_*: A \to \K(\vec y)$ is then just taking the
    constant term in $\vec w$, so $\pi_*(\Phi(f)) = \varphi_0 = \Phi(\pi_*(f))$.
\end{proof}

\begin{remark}\label{rem:phi-localizes-generality}
    For transforms of polyhedra, the conclusion of \Cref{thm:phi-localizes} can also be seen as
    follows. If $P$ is a pointed rational polyhedron with $\ker \pi \cap \rec(P) = \{0\}$, 
    let $q$ be the product of the $1 - \vec x^{\vec u}$ over the primitive ray
    generators $\vec u$ of $\rec(P)$. By \Cref{lemma:polyhedron-denominator},
    $p := q\,\varsigma(P; \vec x)$ is a Laurent polynomial, so $q \cdot \sigma(P; \vec x) = p$
    holds also on the level of Laurent series. Indeed, if $r \neq 0$ is a denominator of $\sigma(P; \vec x)$, then
    $r\,(q \cdot \sigma(P; \vec x) - p) = 0$, and the Laurent series supported in $Q + \rec(P)$ for
    some polytope $Q$ form an integral domain. By assumption we have $\pi_*(q) \neq 0$, so we get
    $\pi_*(p) = \pi_*(q)\,\pi_*(\sigma(P; \vec x))$, which is the whole statement of \Cref{thm:phi-localizes}
    in this setting. 
    We state \Cref{thm:phi-localizes} in full generality since it may be of
    independent interest.
\end{remark}

\subsection{Integer-point transforms}\label{subsec:transforms}

We now specialize to the Laurent series that arise from polyhedra.

\begin{definition}
    Let $\cV$ be the $\K$-vector space of functions $\R^n \to \K$ spanned by the indicator functions
    $\mathbf{1}_P$ of rational polyhedra $P \subseteq \R^n$, whose elements we call \emph{convex chains}.

    Denote by $\sigma: \cV \to \K\llbracket \vec x^{\pm 1} \rrbracket$ the $\K$-linear map
    \[
        \sigma(\varphi; \vec x) := \sum_{\vec a \in \Z^n} \varphi(\vec a)\, \vec x^{\vec a}.
    \]
    For a rational polyhedron $P$ we write
    \[
        \sigma(P; \vec x) := \sigma(\mathbf{1}_P; \vec x) = \sum_{\vec a \in P \cap \Z^n} \vec x^{\vec a}
    \]
    and call it the \emph{integer-point transform} of $P$.
\end{definition}

For every rational polyhedron $P$ the series $\sigma(P; \vec x)$ is summable \cite[Theorem 3.1]{Barvinok1999AnAT}, which allows us to define the rational function associated to $P$.
\begin{definition}\label{def:varsigma}
    Denote by
    \[
        \varsigma := \Phi \circ \sigma: \cV \to \K(\vec x)
    \]
    the $\K$-linear map sending $\varphi \in \cV$ to the rational function
    represented by the series $\sigma(\varphi; \vec x)$. For a rational polyhedron $P$ we write
    $\varsigma(P; \vec x) := \varsigma(\mathbf{1}_P; \vec x)$. 
\end{definition}

If $P$ contains a line, then $\varsigma(P; \vec x)$ is $0$. Indeed, for a nonzero lattice
vector $\vec u$ in the lineality space of $P$ we have $P + \vec u = P$, so that
$(1 - \vec x^{\vec u})\,\sigma(P; \vec x) = 0$.
We say that $\varphi, \psi \in \cV$
agree \emph{modulo lineality}, and write $\varphi \equiv \psi$, if their difference is a linear
combination of indicator functions of polyhedra containing a line.
It follows that elements of $\cV$ that agree modulo lineality have the same image under
$\varsigma$.

Suppose $P$ is a rational polyhedron with $P \cap \Z^n \neq \emptyset$.
Then $\sigma(P; \vec x) \in \cF_{\pi}$ if and only if
$\rec(P) \cap \ker \pi = \{0\}$, where
$\rec(P)$ is the recession cone of $P$.
In this case the series
\[\sigma_\pi(P; \vec y) := \pi_* \sigma(P; \vec x)\]
is well defined and its coefficients are the fibre counts
$\#\left(P \cap \pi^{-1}(\vec b) \cap \Z^n\right)$ for $\vec b \in \Z^m$.
By \Cref{thm:phi-localizes}, $\sigma_\pi(P; \vec y)$ is summable.

The following lemma gives a denominator for $\varsigma(P; \vec x)$, and with it a sufficient
condition for $\varsigma(P; \vec x)$ to be pushed forward along the projection $\pi$.

\begin{lemma}\label{lemma:polyhedron-denominator}
    Let $P \subseteq \R^n$ be a pointed rational polyhedron and let $D$ be the set of primitive ray
    generators of the rays of $\rec(P)$. Then there exists a Laurent polynomial
    $p \in \K[\vec x^{\pm 1}]$ such that
    \[
        \varsigma(P; \vec x) = \frac{p(\vec x)}{\prod_{\vec u \in D} (1 - \vec x^{\vec u})}.
    \]
\end{lemma}

\begin{proof}
    For a cone, the claim follows by triangulating it without adding new rays and
    applying the formula for simplicial cones \cite[Example 3.3]{Barvinok1999AnAT}. By Brion's
    theorem, $\varsigma(P; \vec x)$ is a sum of terms coming from its vertex cones.
    By the cone case, $\varsigma(P; \vec x)$ can therefore be written over the product of the
    $1 - \vec x^{\vec u}$, with $\vec u$ running over the primitive edge directions of $P$ up to sign,
    as $1 - \vec x^{-\vec u} = -\vec x^{-\vec u}(1 - \vec x^{\vec u})$. It
    suffices to show that $\varsigma(P; \vec x)$ can be written with a denominator coprime
    to $1 - \vec x^{\vec u}$ for primitive $\vec u$ with $\pm \vec u \notin D$. Every edge of $P$ parallel to such a $\vec u$ is bounded.
    Choose a rational functional $\mu$ that is positive on $\rec(P) \setminus \{0\}$ with
    $\mu(\vec u) \neq 0$, and a $c$ so large that all these edges lie in $\{\mu < c\}$. Cutting along
    $H = \{\mu = c\}$ yields
    \[\varsigma(P; \vec x) = \varsigma(P \cap \{\mu \geq c\}; \vec x) + \varsigma(P \cap \{\mu \leq c\}; \vec x)
    - \varsigma(P \cap H; \vec x).\] 
    The polyhedra $P \cap \{\mu \leq c\}$ and $P \cap H$ are bounded, so their transforms are Laurent polynomials. No edge of
    $P \cap \{\mu \geq c\}$ is parallel to $\vec u$, so again by Brion's theorem and the cone case,
    we deduce that $\varsigma(P \cap \{\mu \geq c\}; \vec x)$ has no factor $1 - \vec x^{\vec u}$ in its denominator.
\end{proof}

Whenever $\varsigma(P; \vec x) \in \K[\vec x^{\pm 1}]_{\ker \pi_*}$, we set
\[
    \varsigma_\pi(P; \vec y) := \pi_*\left(\varsigma(P; \vec x)\right) \in \K(\vec y).
\]
By \Cref{lemma:polyhedron-denominator} this holds whenever $P$ is pointed and no ray of $\rec(P)$
is contracted by $\pi$.
This is defined in cases where $\sigma_\pi(P; \vec y)$ is not, see \Cref{ex:infinite-fibres} below.
Whenever both $\varsigma_\pi(P; \vec y)$ and $\sigma_\pi(P; \vec y)$ are defined, the two agree by \Cref{thm:phi-localizes}.

\begin{example}\label{ex:infinite-fibres}
    Let $n = 2$, $m = 1$ and $\pi(a_1, a_2) = a_1$, so that $\ker \pi_* = (1 - x_2)$, and let
    $C = \Cone((1, 1), (-1, 1))$. The fibre of $C$ over $b \in \Z$ is the ray
    $\{(b, a_2) : a_2 \geq |b|\}$, which contains infinitely many lattice points, so
    $\sigma_\pi(C; y)$ is not defined. On the other hand,
    \[
        \varsigma(C; \vec x) = \frac{1 + x_2}{(1 - x_1 x_2)(1 - x_1^{-1} x_2)},
    \]
    which has a well-defined pushforward
    \[
        \varsigma_\pi(C; y) = \frac{2}{(1 - y)(1 - y^{-1})}.
    \]

    One may interpret this rational function as follows. Let
    $C' := \Cone((1, 1), (1, -1)) \setminus \R_{\geq 0}(1, 1)$ be a half-open cone. The cones $C$ and $C'$ together make up the half-plane
    $\{a_1 + a_2 \geq 0\}$, so that
    $\mathbf{1}_C \equiv -\mathbf{1}_{C'}$,
    and $\varsigma(C; \vec x) = -\varsigma(C'; \vec x)$.
    Unlike $C$, the half-open cone $C'$ has finite fibres.
    Over $b \geq 1$ it contains the $2b$ lattice points $(b, a_2)$ with
    $-b \leq a_2 \leq b - 1$, and over $b \leq 0$ none. It follows that
    \[
        \varsigma_\pi(C'; y) = \Phi\Bigl(\sum_{b \geq 1} 2b\, y^b\Bigr) = \frac{2y}{(1 - y)^2},
    \]
    so that $\varsigma_\pi(C; y) = -\varsigma_\pi(C'; y)$, in agreement with the expression above
    (\Cref{fig:infinite-fibres}).

    \begin{figure}[ht]
    \centering
    \begin{tikzpicture}[scale=0.5, font=\small, line join=round, baseline=0]
        \foreach \i in {-3, ..., 3} \foreach \j in {-3, ..., 3} \fill[black!35] (\i, \j) circle (0.06);
        \fill[red!60!black!10] (0, 0) -- (-3.4, 3.4) -- (3.4, 3.4) -- cycle;
        \foreach \i in {-3, ..., 3} \draw[thin, dotted, black!45] (\i, 3.4) -- (\i, -4.6);
        \draw[semithick, red!40!black] (-3.4, 3.4) -- (0, 0) -- (3.4, 3.4);
        \foreach \i in {-3, ..., 3} \foreach \j in {-3, ..., 3} {
            \pgfmathtruncatemacro{\ai}{abs(\i)}
            \ifnum\j<\ai \else \fill (\i, \j) circle (0.09); \fi
        }
        \node[fill=red!60!black!10, inner sep=1pt] at (-0.5, 2.5) {$C$};
        \draw[->, semithick] (4.2, 1.4) -- node[right] {$\pi$} (4.2, -2.7);
        \draw[thin] (-3.6, -4.6) -- (3.6, -4.6);
        \foreach \i in {-3, ..., 3} {
            \fill (\i, -4.6) circle (0.09);
            \node[below=2pt] at (\i, -4.6) {$\infty$};
        }
        \node[right] at (3.6, -4.6) {$\R$};
    \end{tikzpicture}
    \hspace{2em}
    \begin{tikzpicture}[scale=0.5, font=\small, line join=round, baseline=0]
        \foreach \i in {-3, ..., 3} \foreach \j in {-3, ..., 3} \fill[black!35] (\i, \j) circle (0.06);
        \fill[red!60!black!10] (0, 0) -- (3.4, 3.4) -- (3.4, -3.4) -- cycle;
        \foreach \i in {-3, ..., 3} \draw[thin, dotted, black!45] (\i, 3.4) -- (\i, -4.6);
        \draw[semithick, red!40!black] (0, 0) -- (3.4, -3.4);
        \draw[semithick, dashed, red!40!black] (0, 0) -- (3.4, 3.4);
        \foreach \i in {1, ..., 3} \foreach \j in {-3, ..., 3} {
            \pgfmathtruncatemacro{\aj}{abs(\j)}
            \ifnum\aj<\i \fill (\i, \j) circle (0.09); \else
                \ifnum\j=-\i \fill (\i, \j) circle (0.09); \fi \fi
        }
        \node[fill=red!60!black!10, inner sep=1pt] at (2.5, -0.5) {$C'$};
        \draw[->, semithick] (4.2, 1.4) -- node[right] {$\pi$} (4.2, -2.7);
        \draw[thin] (-3.6, -4.6) -- (3.6, -4.6);
        \foreach \i/\n in {-3/0, -2/0, -1/0, 0/0, 1/2, 2/4, 3/6} {
            \fill (\i, -4.6) circle (0.09);
            \node[below=2pt] at (\i, -4.6) {$\n$};
        }
        \node[right] at (3.6, -4.6) {$\R$};
    \end{tikzpicture}
    \caption{The cone $C$ (left) and the half-open cone $C'$ (right) of \Cref{ex:infinite-fibres}.
    Every fibre of $\pi(a_1, a_2) = a_1$
    meets $C$ in infinitely many lattice points, but only in finitely many for $C'$.}
    \label{fig:infinite-fibres}
    \end{figure}
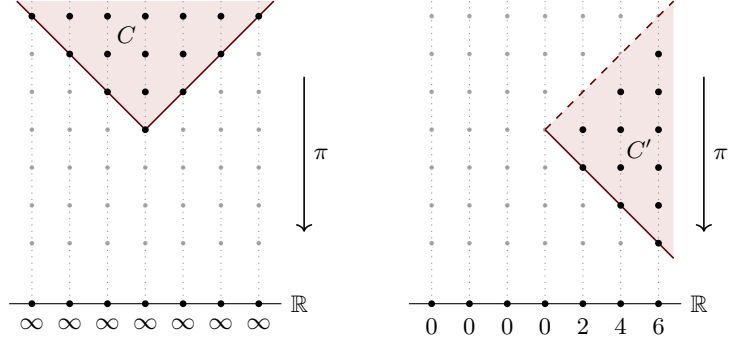
\end{example}

The map $\varsigma_\pi$ sends every polyhedron containing a line to $0$. The following example shows
that it can also vanish on pointed cones.

\begin{example}[Cones over circuits]\label{ex:circuit}
    If $\pi$ kills no ray of a pointed rational cone $C$, then $\varsigma_\pi(C; \vec y)$ is well defined
    (\Cref{lemma:polyhedron-denominator}), but it
    can vanish, as it does for suitable cones over circuits. Let $C \subseteq \R^n$ be a
    full-dimensional pointed rational cone with exactly $n + 1$ rays, with primitive generators
    $\vec u_0, \dots, \vec u_n$. They satisfy a linear relation $\sum_i \lambda_i \vec u_i = 0$,
    unique up to scaling.
    Grouping the rays by sign gives simplicial cones
    \[
        C_+ := \Cone(\vec u_i : \lambda_i \geq 0), \qquad C_- := \Cone(\vec u_i : \lambda_i < 0)
    \]
    with $C = C_+ + C_-$, and $C_+ \cap C_- = \R_{\geq 0} \vec w$, where
    $\vec w$ is the primitive lattice vector on the ray through
    $\sum_{\lambda_i > 0} \lambda_i \vec u_i = \sum_{\lambda_i < 0} (-\lambda_i) \vec u_i$. If
    $C \cap \Z^n = (C_+ \cap \Z^n) + (C_- \cap \Z^n)$, then
    \[
        \varsigma(C; \vec x) = (1 - \vec x^{\vec w})\, \varsigma(C_+; \vec x)\, \varsigma(C_-; \vec x)
    \]
    by the formula for free sums \cite[Corollary 2.5]{Bruns15}, see also
    \cite{BeckJayawantMcAllister13}. If moreover $\pi(\vec w) = 0$ and $\pi(\vec u_i) \neq 0$ for all $i$,
    then $\varsigma_\pi(C; \vec y) = 0$. Indeed, as $C_\pm$ is simplicial, $\varsigma(C_\pm; \vec x)$ can be
    written over the product of the factors $1 - \vec x^{\vec u_i}$ ranging over its rays
    \cite[Example 3.3]{Barvinok1999AnAT}, which avoids the prime $\ker \pi_*$, while
    $\pi_*(1 - \vec x^{\vec w}) = 0$.

    For instance, let
    \[
        C = \Cone((1, 0, 1), (-1, 0, 1), (0, 1, 1), (0, -1, 1))
    \]
    and $\pi$ the projection onto the first two coordinates. The
    relation between the primitive ray generators of $C$ is
    \[
        (1, 0, 1) + (-1, 0, 1) = (0, 1, 1) + (0, -1, 1),
    \]
    so $C_+ = \Cone((1, 0, 1), (-1, 0, 1))$, $C_- = \Cone((0, 1, 1), (0, -1, 1))$ and
    $\vec w = (0, 0, 1) \in \ker \pi$ (see \Cref{fig:circuit}). A lattice point $\vec a$ of $C$ has $a_3 \geq |a_1| + |a_2|$, and so may be written as the sum of $(a_1, 0, |a_1|) \in C_+$ and
    $(0, a_2, a_3 - |a_1|) \in C_-$. It follows that $C \cap \Z^3 = (C_+ \cap \Z^3) + (C_- \cap \Z^3)$, so
    \[
        \varsigma(C; \vec x) = (1 - x_3)\frac{1 + x_3}{(1 - x_1x_3)(1 - x_1^{-1}x_3)}
        \frac{1 + x_3}{(1 - x_2x_3)(1 - x_2^{-1}x_3)}
    \]
    and $\varsigma_\pi(C; \vec y) = 0$.
    The series
    $\sigma_\pi(C; \vec y)$ is not defined, as every fibre of $\pi$ meets $C$ in a half-line.

    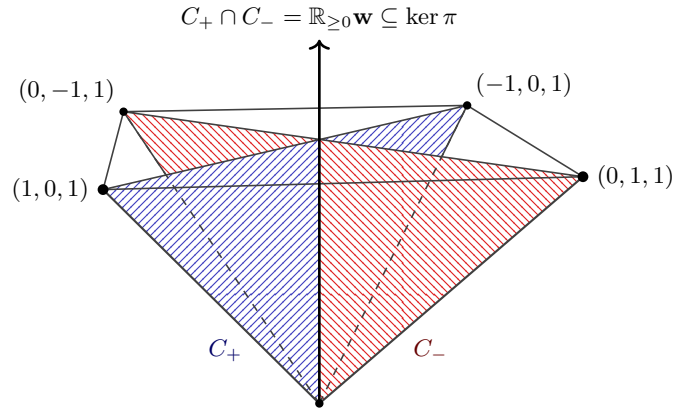
\begin{figure}[ht]
    \centering
    \begin{tikzpicture}[scale=4.6, line join=round, font=\small]
        \coordinate (O)   at (0, -0.433);
        \coordinate (g1)  at (-0.391, -0.046);  
        \coordinate (g2)  at (0.485, -0.016);   
        \coordinate (g3)  at (0.303, 0.178);    
        \coordinate (g4)  at (-0.397, 0.158);   
        \coordinate (r1)  at (-0.622, 0.182);
        \coordinate (r2)  at (0.758, 0.219);
        \coordinate (r3)  at (0.425, 0.424);
        \coordinate (r4)  at (-0.563, 0.406);
        \coordinate (Z)   at (0, 0.613);
        \coordinate (w)   at (0, 0.3255);
        \fill[pattern=north west lines, pattern color=red!75!gray] (O) -- (r4) -- (w) -- cycle;
        \draw[semithick, black!75] (O) -- (r4);
        \fill[black!0] (O) -- (r1) -- (w) -- cycle;
        \fill[pattern=north east lines, pattern color=blue!50!gray] (O) -- (r1) -- (r3) -- cycle;
        \draw[semithick, black!75] (O) -- (r3);
        \fill[black!0] (O) -- (r2) -- (w) -- cycle;
        \fill[pattern=north west lines, pattern color=red!75!gray] (O) -- (r2) -- (w) -- cycle;
        \draw[semithick, black!75] (r1) -- (r2);
        \draw[semithick, black!75] (r4) -- (r2) -- (r3) -- (r4) -- (r1) -- (r3);
        \draw[thick, black!75] (O) -- (r1);
        \draw[thick, black!75] (O) -- (r2);
        \draw[semithick, dashed, black!75] (O) -- (r3);
        \draw[semithick, dashed, black!75] (O) -- (r4);
        \draw[->, line width=1pt] (O) -- (Z);
        \fill (O) circle (0.012);
        \fill (r1) circle (0.0155);
        \fill (r2) circle (0.0155);
        \fill (r3) circle (0.012);
        \fill (r4) circle (0.012);
        \node[left, inner sep=5pt] at (r1) {$(1, 0, 1)$};
        \node[right, inner sep=5pt] at (r2) {$(0, 1, 1)$};
        \node[above right, inner sep=3pt] at (r3) {$(-1, 0, 1)$};
        \node[above left, inner sep=3pt] at (r4) {$(0, -1, 1)$};
        \node[above] at (Z) {$C_+ \cap C_- = \R_{\geq 0}\vec w \subseteq \ker \pi$};
        \node[blue!40!black] at (-0.27, -0.28) {$C_+$};
        \node[red!40!black] at (0.32, -0.28) {$C_-$};
    \end{tikzpicture}
    \caption{The cone over a circuit of \Cref{ex:circuit}, the two simplicial cones $C_+$ and $C_-$ into which
    the relation splits it, and their common ray, which $\pi$ kills.}
    \label{fig:circuit}
    \end{figure}
\end{example}

\section{Projected Brion's theorem}\label{sec:brion}

Brion's theorem states that for any pointed rational polyhedron $P$, 
we have
\[
    \varsigma(P; \vec x) = \sum_{v \textrm{ vertex of } P} \varsigma(K_v; \vec x),
\]
where $K_v := v + \R_{\geq 0}(P - v)$ are the vertex cones of $P$.
Brion proved this for lattice polytopes \cite{Brion88}, and the form above is
\cite[Theorem 3.5]{Barvinok1999AnAT}.

The summands cannot be pushed forward one at a time, as $\varsigma(K_v; \vec x)$ need not lie in
$\K[\vec x^{\pm 1}]_{\ker \pi_*}$, so they have to be grouped. The next example shows how.

\begin{example}\label{ex:pyramid-poles}
    Take the pyramid $P = \Conv\{\pm \vec e_1, \vec e_2, \pm \vec e_3\}$ depicted in \Cref{fig:pyramid}.
    Let $\pi(\vec a) = a_1$, so that
    $\ker \pi_* = (1 - x_2, 1 - x_3)$.
    The rational function $\varsigma(K_{\vec e_2}; \vec x)$ has denominator
    \[
            (1 - x_1x_2^{-1})(1 - x_1^{-1}x_2^{-1})(1 - x_3x_2^{-1})(1 - x_3^{-1}x_2^{-1}).
    \]
    The factors $1 - x_3x_2^{-1}$ and $1 - x_3^{-1}x_2^{-1}$ both lie in $\ker \pi_*$, so 
    $\varsigma(K_{\vec e_2}; \vec x)$ cannot be pushed forward along $\pi$.
    These factors come from edges whose direction is killed by $\pi$, which suggests adding the rational functions of the vertex cones at vertices joined by such an edge. Indeed, a quick computation shows that adding 
    $\varsigma(K_{\vec e_3}; \vec x)$ to $\varsigma(K_{\vec e_2}; \vec x)$ cancels the factor $1 - x_3x_2^{-1}$ of the denominator, which corresponds to the edge $[\vec e_2, \vec e_3]$. The other problematic factor
    then disappears after adding $\varsigma(K_{-\vec e_3}; \vec x)$.

    \begin{figure}[ht]
    \centering
    \begin{tikzpicture}[scale=1.8, font=\small, line join=round]
        \coordinate (A) at ( 0.000,  0.942);
        \coordinate (P) at (-0.357, -0.254);
        \coordinate (M) at ( 0.298,  0.212);
        \coordinate (Y) at ( 0.923, -0.087);
        \coordinate (B) at (-0.864,  0.081);
        \draw[thin, black!40] (0.35, 0.265) -- (0.6, 0.55);
        \fill[very nearly transparent, red!60!black] (P) -- (A) -- (M) -- cycle;
        \draw[thin, dashed, black!70] (M) -- (Y) (M) -- (B);
        \draw[line width=1.4pt, dashed, red!40!black] (A) -- (M);
        \draw[thick] (P) -- (Y) (P) -- (B) (A) -- (Y) (A) -- (B);
        \draw[line width=1.6pt, red!40!black] (A) -- (P);
        \fill (Y) circle (0.04);
        \fill (B) circle (0.038);
        \fill[red!40!black] (A) circle (0.04);
        \fill[red!40!black] (P) circle (0.043);
        \fill[red!40!black] (M) circle (0.036);
        \node[above=2pt] at (A) {$\vec e_2$};
        \node[below=2pt] at (P) {$\vec e_3$};
        \node[right=2pt] at (Y) {$\vec e_1$};
        \node[left=2pt] at (B) {$-\vec e_1$};
        \node[anchor=south west, inner sep=1pt] at (0.5, 0.55) {$-\vec e_3$};
    \end{tikzpicture}
    \caption{The pyramid of \Cref{ex:pyramid-poles}, with the two edges contracted by
    $\pi$ and the shaded triangle spanned by the contracted class $\{\vec e_2, \pm \vec e_3\}$, which is
    not a face.}
    \label{fig:pyramid}
    \end{figure}
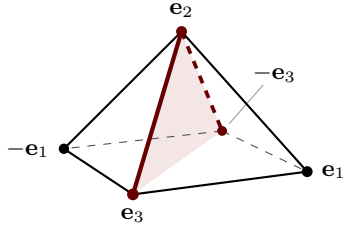
\end{example}

The above example motivates the grouping of vertices into so-called \emph{contracted classes}.

\begin{definition}\label{def:contracted-classes}
    Let $P \subseteq \R^n$ be a rational polyhedron with vertex set $V(P)$.
    A face $F$ of $P$ is \emph{contracted} by $\pi$ if $\pi$ kills its direction space, i.e., if
    $\pi(F - F) = 0$.
    Let $\sim$ be the equivalence relation on $V(P)$ generated by $v_1 \sim v_2$ whenever
    $[v_1, v_2]$ is an edge of $P$ contracted by $\pi$. Its equivalence classes are the
    \emph{contracted classes} of $P$.
\end{definition}

The endpoints of a contracted edge have the same image under $\pi$, so each contracted class lies in
a single fibre of $\pi$. When $P$ is simple and no unbounded edge of $P$ is contracted, every contracted class of $P$ is
the vertex set of a bounded face of $P$. In general this fails, as the pyramid shows.

\begin{example}\label{ex:pyramid-classes}
   In \Cref{ex:pyramid-poles}, the only contracted edges are those joining $\vec e_2$ with $\pm \vec e_3$, so the contracted classes are the singletons $\{\vec e_1\}$ and $\{-\vec e_1\}$ and the class $S := \{\vec e_3, \vec e_2, -\vec e_3\}$. The triangle spanned by $S$ is not a face of $P$ (see \Cref{fig:pyramid}). 
\end{example}

Our main result is that the rational functions of the vertex cones of a contracted class sum to an element of
$\K[\vec x^{\pm 1}]_{\ker \pi_*}$, so that $\pi_*$ may be applied class by class. This is the
projected form of Brion's theorem. 

\begin{theorem}[Projected Brion's theorem]\label{thm:projected-brion}
    Let $P \subseteq \R^n$ be a pointed rational polyhedron such that no unbounded edge of $P$ is
    contracted by $\pi$. Then for every contracted class $S$ of $P$
    (\Cref{def:contracted-classes}),
    \[
        \sum_{v\in S} \varsigma(K_v; \vec x) \in \K[\vec x^{\pm 1}]_{\ker \pi_*},
    \]
    so that $\pi_*$ may be applied to it, and
    \[
        \varsigma_\pi(P; \vec y) = \sum_{S \in V(P)/\sim} \pi_* \left(
            \sum_{v \in S} \varsigma(K_v; \vec x)
        \right) \in \K(\vec y),
    \]
    the outer sum being finite as $V(P)$ is. If moreover $\rec(P) \cap \ker \pi = \{0\}$, this is
    the rational function of the series $\sigma_\pi(P; \vec y)$, whose coefficients are the fibre
    counts of $P$.
\end{theorem}

\begin{example}\label{ex:pyramid}
    The lattice points of the pyramid $P$ of \Cref{ex:pyramid-poles} are its five vertices and the origin, so that the fibre counts are
    $1, 4, 1$. Taking the classes in the order $\{-\vec e_1\}$, $S$, $\{\vec e_1\}$,
    \Cref{thm:projected-brion} reads
    \[
        \varsigma_\pi(P; y)
        = \frac{1 + y}{y\,(1 - y)^3}
          - \frac{8\,y}{(1 - y)^2}
          - \frac{y^3(1 + y)}{(1 - y)^3}
        = y^{-1} + 4 + y.
    \]
\end{example}

As the following example shows, the vertices in the same fibre of $\pi$ might not belong to the same contracted class and hence do not have to be grouped before applying $\pi_*$.

\begin{example}\label{ex:cross-polytope}
    Let $n = 2$, $m = 1$ and
    $\pi(a_1, a_2) = a_1$, so that $\ker \pi_* = (1 - x_2)$, and let
    $P = \Conv\{\pm \vec e_1, \pm \vec e_2\}$ be the cross-polytope. Its
    edges have directions $\pm \vec e_1 \pm \vec e_2$, so none is contracted and every vertex is a
    contracted class of its own. Even though $\pm \vec e_2$ lie in one fibre of $\pi$, there is no reason to join them. Indeed,
    \[
        \varsigma(K_{\vec e_2}; \vec x) = \frac{1 + x_2}{(1 - x_1 x_2^{-1})(1 - x_1^{-1} x_2^{-1})} ,
        \qquad
        \varsigma(K_{-\vec e_2}; \vec x) = \frac{1 + x_2^{-1}}{(1 - x_1 x_2)(1 - x_1^{-1} x_2)}
    \]
    lie in $\K[\vec x^{\pm 1}]_{\ker \pi_*}$ already, both with image $-2y/(1 - y)^2$.
\end{example}

Before proving \Cref{thm:projected-brion}, we show in \Cref{lemma:edge-cancellation} that the
denominators of the terms $\varsigma(K_v; \vec x)$ cancel along edges.
This is measured by orders of vanishing.
For any primitive vector $\vec u \in \Z^n$, the polynomial $1 - \vec x^{\vec u}$ is irreducible,
because a change of basis carries $\vec u$ to $\vec e_1$ and $1 - x_1$ is irreducible.
Since $\K[\vec x^{\pm 1}]$ is a UFD, this determines a valuation
\[
    \ord_{1 - \vec x^{\vec u}}: \K[\vec x^{\pm 1}] \to \N \cup \{\infty\},
\]
with $\ord_{1 - \vec x^{\vec u}}(0) = \infty$, and this extends to a
valuation $\ord_{1 - \vec x^{\vec u}}: \K(\vec x) \to \Z \cup \{\infty\}$. 
Note that $\ord_{1 - \vec x^{\vec u}}(f) \geq 0$ says precisely that $f$ can be written with a denominator
coprime to $1 - \vec x^{\vec u}$. Finally,
$1 - \vec x^{-\vec u} = -\vec x^{-\vec u}(1 - \vec x^{\vec u})$, so $1 - \vec x^{\vec u}$ and
$1 - \vec x^{-\vec u}$ are associates and $\ord_{1 - \vec x^{\vec u}}$ only depends on $\vec u$ up
to sign.

\begin{lemma}\label{lemma:edge-cancellation}
    Let $P \subseteq \R^n$ be a rational polyhedron and let $v_1, v_2$ be two vertices spanning an
    edge $e$ of $P$, with primitive direction $\vec u$. Then
    \[
        \ord_{1 - \vec x^{\vec u}}
        \left(\varsigma(K_{v_1}; \vec x) + \varsigma(K_{v_2}; \vec x)\right) \geq 0.
    \]
\end{lemma}

\begin{proof}
    We may assume that  $\vec u \in \rec(K_{v_1})$ and $-\vec u \in \rec(K_{v_2})$.
    Let $T_e := K_{v_1} + \R\vec u = K_{v_2} + \R\vec u$ be the tangent cone of $P$ along $e$, and let
    $\rho \colon \Z^n \to \Z^n/\Z\vec u$ be the quotient map, so that $\ker \rho_*$ is generated by
    $1 - \vec x^{\vec u}$. Every line parallel to $\vec u$ through a lattice point of $T_e$ meets
    $K_{v_1} \setminus (K_{v_1} + \vec u)$ in exactly one lattice point, and every lattice point of $K_{v_1}$ lies on such a line.
    Hence $(1 - \vec x^{\vec u})\,\sigma(K_{v_1}; \vec x)$ lies in $\cF_\rho$, and its pushforward is
    $f := \sum_{\vec b \in \rho(T_e \cap \Z^n)} \vec y^{\vec b}$.
    By \Cref{thm:phi-localizes},
    $(1 - \vec x^{\vec u})\,\varsigma(K_{v_1}; \vec x)$ lies in $\K[\vec x^{\pm 1}]_{\ker \rho_*}$ with
    pushforward $\Phi(f)$.
    The same argument shows that
    $(1 - \vec x^{\vec u})\,\varsigma(K_{v_2}; \vec x)$ has pushforward $-\Phi(f)$,
    so that the function
    \[(1 - \vec x^{\vec u})\left(\varsigma(K_{v_1}; \vec x) + \varsigma(K_{v_2}; \vec x)\right)\]
    lies in the kernel of $\rho_*$, which is generated by $1 - \vec x^{\vec u}$.
    Dividing by $1 - \vec x^{\vec u}$ gives the claim.
\end{proof}

Summed over a contracted class, these cancellations prove \Cref{thm:projected-brion}.

\begin{proof}[Proof of \Cref{thm:projected-brion}]
    For a contracted class $S$, let
    \[
        \psi_S(\vec x) := \sum_{v \in S} \varsigma(K_v; \vec x).
    \]
    It suffices to show that $\psi_S \in \K[\vec x^{\pm 1}]_{\ker \pi_*}$, the rest following from Brion's theorem, linearity of the pushforward $\pi_*$ and \Cref{thm:phi-localizes}.

    Let $E_S$ hold one primitive direction vector for each edge direction occurring at a vertex of $S$, taken up to sign, so that the
    $1 - \vec x^{\vec u}$ for $\vec u \in E_S$ are pairwise non-associate irreducibles. Then by
    \Cref{lemma:polyhedron-denominator}, applied to the vertex cones $K_v$, there exists some Laurent polynomial $p$ such that
    \[
        \psi_S(\vec x) = \frac{p(\vec x)}{\prod_{\vec u \in E_S} (1 - \vec x ^{\vec u})}.
    \]
    Let $\vec u \in E_S$ be a direction with $\pi(\vec u) = 0$, so that
    $1 - \vec x^{\vec u} \in \ker \pi_*$, and let $S_{\vec u} \subseteq S$ be the set of vertices of
    $S$ lying on an edge of $P$ parallel to $\vec u$. Such an edge is bounded, as no unbounded edge
    of $P$ is contracted, and its other endpoint is again a vertex in $S_{\vec u}$.
    Moreover, a vertex lies on at most one edge parallel to
    $\vec u$. These edges therefore pair up the vertices of $S_{\vec u}$, and we may apply
    \Cref{lemma:edge-cancellation} to each pair. For $v \in S \setminus S_{\vec u}$
    no ray of $K_v$ is parallel to $\vec u$, so
    $\ord_{1 - \vec x^{\vec u}}(\varsigma(K_v; \vec x)) \geq 0$ by \Cref{lemma:polyhedron-denominator}.
    We deduce that
    \[
        \ord_{1 - \vec x^{\vec u}} \psi_S \geq 0 .
    \]
    It follows that $\psi_S$ may be
    written over the factors $1 - \vec x^{\vec u}$ with $\pi(\vec u) \neq 0$, each of which avoids the prime ideal
    $\ker \pi_*$.
\end{proof}

In fact, the grouping in \Cref{thm:projected-brion} is minimal, at least for lattice polyhedra.
That is, every set of summands in Brion's formula whose sum lies in
$\K[\vec x^{\pm 1}]_{\ker \pi_*}$ is a union of contracted classes.

\begin{proposition}\label{prop:minimality}
    Let $P \subseteq \R^n$ be a rational polyhedron such that every contracted edge of $P$ lies on a
    line through a lattice point (for instance a lattice polyhedron). 
    If a set $S$ of vertices of $P$
    satisfies $\psi_S := \sum_{v \in S} \varsigma(K_v; \vec x) \in \K[\vec x^{\pm 1}]_{\ker \pi_*}$,
    then $S$ is a union of
    contracted classes.
\end{proposition}

\begin{proof}
    Suppose some contracted edge, with primitive direction $\vec u$, has exactly one endpoint in $S$,
    and let $\rho\colon \Z^n \to \Z^n/\Z\vec u$ be the quotient map. Since $1 - \vec x^{\vec u}$ is a
    prime in $\ker \pi_*$, we have $\ord_{1 - \vec x^{\vec u}} \psi_S \geq 0$, so
    \[
        \sum_{v \in S} \rho_*\bigl((1 - \vec x^{\vec u})\varsigma(K_v; \vec x)\bigr)
        = \rho_*\bigl((1 - \vec x^{\vec u})\psi_S\bigr) = 0.
    \]
    By the proof of \Cref{lemma:edge-cancellation}, a vertex without an edge parallel to $\vec u$
    contributes nothing to this sum, the two endpoints of such an edge $e$ contribute
    $\pm\varsigma(C_e; \vec y)$ with opposite signs, where $C_e$ is the image of the tangent cone along $e$ in
    $\R^n/\R\vec u$, and pairs inside $S$ cancel. So $\sum_{e} \pm\varsigma(C_e; \vec y) = 0$, the sum running
    over the edges $e$ parallel to $\vec u$ with exactly one endpoint in $S$, with sign $+$ if $e$
    leaves that endpoint in direction $\vec u$ and $-$ otherwise. The sum is nonempty. The cone $C_e$
    is the vertex cone of the projection of $P$ to $\R^n/\R\vec u$ at the image $q_e$ of $e$, and
    distinct edges have distinct images. Choose a linear functional $\ell$ on $\Z^n/\Z\vec u$ that is
    minimized on this projection only at one of the $q_e$ in the sum, say $q$, and is nonzero on all
    rays of the $C_e$. Triangulating each $C_e$ without new rays and flipping the rays on which
    $\ell < 0$ writes $\varsigma(C_e; \vec y)$ as a signed sum of transforms of half-open cones with
    apex $q_e$ and all rays in $\{\ell > 0\}$, closed if nothing is flipped. Pushing forward along $\ell$
    (\Cref{thm:phi-localizes}), the cones at $q_e$ contribute only in degrees at least $\ell(q_e)$,
    and $\ell(q_e) > \ell(q)$ for $q_e \neq q$. The cone at $q$ is not flipped, and $q$ is a lattice
    point by hypothesis, so the coefficient in degree $\ell(q)$ is $\pm 1$, a contradiction.
    Hence $S$ is a union of contracted classes.
\end{proof}

The hypothesis cannot be dropped. For $P = [1/3, 2/3]^2 \times [0, 1]$ and $\pi(\vec a) = a_3$ the
bottom four vertices form one class, their terms are
$\pm x_1 x_2/((1 - x_1)(1 - x_2)(1 - x_3))$, and both ways of pairing them along edges give sums
equal to $0$. For instance $S = \{(1/3, 1/3, 0), (2/3, 1/3, 0)\}$ has $\psi_S = 0$ but is not a union of
contracted classes.

The contracted classes also carry geometric meaning. For a lattice polytope they are the
vertex sets corresponding to the connected components of the fixed locus of the associated toric variety under the action of a subtorus, and
\Cref{thm:projected-brion} is the localization formula for that subtorus, with one term per
component.

\begin{remark}[Localization]\label{rem:localization}
    Let $P$ be a full-dimensional lattice polytope, not necessarily smooth, $X$ its projective toric variety with torus
    $T$, and $\cL$ the ample line bundle corresponding to
    $P$. Then $\sigma(P; \vec x)$ is the character of $H^0(X, \cL)$
    \cite[Proposition 4.3.3]{CoxLittleSchenck11}, which is also the equivariant
    Euler characteristic of $\cL$ by Demazure vanishing. Brion's theorem then corresponds to
    the localization formula in
    equivariant $K$-theory for $T$ \cite[Section 2.3]{Brion88}. The $T$-fixed points of $X$ correspond to the
    vertices of $P$, and the fixed point associated with $v$ contributes $\varsigma(K_v; \vec x)$, the character of
    the sections of $\cL$ over the affine chart at $v$.

    The projection $\pi$ corresponds to a subtorus $T' \subseteq T$, and $\pi_*$ restricts
    characters to $T'$. A face $F$ of $P$ corresponds to a $T$-orbit $O_F$ of $X$, which is fixed pointwise by $T'$ exactly when $F$ is contracted. Vertices are joined by contracted faces precisely when they are joined by contracted edges, so the contracted classes are the vertex sets corresponding to the connected components
    of the fixed locus $X^{T'}$.
    Thomason's concentration theorem \cite[Th\'eor\`eme 2.1]{Thomason92}, which also holds for singular
    $X$, gives a formula analogous to the localization formula in equivariant $K$-theory, for the fixed locus $X^{T'}$. Along with Demazure vanishing, this writes the character of $H^0(X, \cL)$ as a sum of contributions of the connected components of $X^{T'}$, each regular along
    $\ker \pi_*$. These contributions are the terms of \Cref{thm:projected-brion}.
\end{remark}

\begin{remark}[Twisted transforms and group rings]\label{rem:twisted}
    Let $\theta: \Z^n \to \K^\times$ be a homomorphism and
    \begin{align*}
    \pi_\theta: \K[\vec x^{\pm 1}] &\to \K[\vec y^{\pm 1}]\\
    \vec x^{\vec a} &\mapsto \theta(\vec a)\vec y^{\pi(\vec a)}
    \end{align*}
    the twisted pushforward. Its kernel is prime, and $\pi_\theta$ extends to the localization at
    that kernel. Call an edge with direction $\vec u$ \emph{contracted by $\pi_\theta$} if
    $\pi(\vec u) = 0$ and $\theta(\vec u) = 1$. Then \Cref{thm:phi-localizes} and
    \Cref{thm:projected-brion} hold with $\pi_*$ replaced by
    $\pi_\theta$ and contraction understood with respect to $\pi_\theta$, with the same proofs.

    Twisted pushforwards describe transforms with values in group rings. Let $\varphi: \Z^n \to G$ be a homomorphism to a finitely generated abelian group $G$, with
    torsion subgroup $H$, and
    $\varphi_*: \vec x^{\vec a} \mapsto \varphi(\vec a) \in \K[G]$, so that
    $\varphi_*\sigma(P; \vec x)$ records the lattice points of $P$ in the fibres of $\varphi$. It is the
    transform of \cite[Theorem 1.8]{EhrhartGroupRings} at $\vec x = \vec 1$, whose variables $\vec x$
    keep the denominators $1 - \varphi(\vec u)\vec x^{\vec u}$ from degenerating. Choose a splitting
    $G = H \oplus \Z^m$ and write $\varphi = (\tau, \pi)$ accordingly. Suppose that $\K$ contains a
    primitive $e$-th root of unity, where $e$ is the exponent of $H$. Then the group $\widehat H$ of
    characters $\chi: H \to \K^\times$ has $|H|$ elements, which are linearly independent by
    Dedekind's lemma, so that $t \mapsto (\chi(t))_\chi$ is an isomorphism
    $\K[H] \cong \K^{\widehat H}$ and
    \[
        \K[G] \cong \prod_{\chi \in \widehat H} \K[\vec y^{\pm 1}], \qquad
        \varphi_* \longmapsto \bigl(\pi_{\chi \circ \tau}\bigr)_{\chi},
    \]
    a product of twisted pushforwards. If no unbounded edge of $P$ is contracted by $\pi$, then by
    the above each component $\pi_{\chi \circ \tau}\varsigma(P; \vec x)$ is given by a
    projected Brion formula. The grouping depends on $\chi$, as an edge with $\pi(\vec u) = 0$ is
    contracted by $\pi_{\chi \circ \tau}$ exactly when $\chi(\tau(\vec u)) = 1$. When
    $\rec(P) \cap \ker \pi = \{0\}$ the components are the rational functions of the twisted fibre
    sums, and inverting $\K[H] \cong \K^{\widehat H}$ recovers the fibre counts of $\varphi$. For $G$
    finite, that is, $m = 0$, a pointed $P$ satisfying the hypothesis is bounded. The trivial character then
    contracts every edge, and its single class gives the lattice point count of $P$, while the other
    characters give finer groupings, and all of them together give the numbers of lattice points of
    $P$ in the fibres of $\varphi$.
\end{remark}

\begin{remark}[Computation]\label{rem:computation}
    The grouping into contracted classes is not needed to compute a short rational representation of
    $\varsigma_\pi(P; \vec y)$. Given any short rational representation of $\varsigma(P; \vec x)$,
    the monomial substitution of \cite[Theorem 2.6]{BarvinokWoods03} produces one for
    $\varsigma_\pi(P; \vec y)$. Its hypothesis, that the image of the substitution is not contained in
    the polar set, is equivalent to $\varsigma(P; \vec x)$ having an image under $\pi_*$, which holds by
    \Cref{lemma:polyhedron-denominator} when no unbounded edge of $P$ is contracted, as then no factor
    of the denominator there lies in $\ker \pi_*$. Their method cancels the
    poles of all terms at once along a generic direction. What \Cref{thm:projected-brion} adds is that
    the cancellation already happens within each contracted class, which is what the proof of
    \Cref{thm:linear-family} exploits.
\end{remark}

\section{Linear families of polytopes}\label{sec:linear-families}

We now apply projected Brion's theorem to linear families of polytopes. The main result of this section is \Cref{thm:linear-family}, which makes explicit how the projected integer-point transform of a linear family depends on its parameter. In this section we assume $\K$ is of characteristic $0$.

For a projection to a point, the lattice point count of a lattice family is a polynomial in the
parameter by McMullen's theorem \cite{McMullen78}, see also \cite[Chapter 18]{Barvinok08}. For general $\pi$, the fibre counts are piecewise quasi-polynomial in the
parameter, as they are values of a vector partition function, and a parametric Brion
formula, with vertices depending linearly on the parameter, can be pushed forward along $\pi$ term by
term when no edge is contracted \cite[Section 3.4]{BrionVergne97}. \Cref{thm:linear-family} extends that
formula to arbitrary $\pi$, with one term for each contracted class.

\subsection{Projected Brion for linear families}\label{sec:linear-family-theorem}

Denote by $\cP(\Q^n)$ the set of rational polytopes in $\R^n$. This may be seen as a subset of the space of convex chains $\cV$. There we have a convolution product
$*$, which satisfies
\[ \mathbf 1_P * \mathbf 1_Q = \mathbf 1_{P + Q} \]
for all rational polyhedra $P, Q$.
Its neutral element is $\mathbf 1_{\{0\}}$, and for every rational polytope $P$, the convex chain
$\mathbf 1_P$ is invertible. Denote by $\cP^*(\Q^n)$ the groupification of $\cP(\Q^n)$ under Minkowski addition, which we may see by the above as a subset of $\cV$ \cite{PukhlikovKhovanskii92}.

A \emph{linear family of polytopes} is a map $\Delta: \Q^r \to \cP^*(\Q^n)$ that is $\Q$-linear for the vector space structure on $\cP^*(\Q^n)$ given by the convolution product and dilation. We further assume that there exists a
full-dimensional rational cone $C \subseteq \Q^r$ such that
$\Delta(C) \subseteq \cP(\Q^n)$.
Then for any $\vec c \in \relint(C)$, the polytopes $\Delta(\vec c)$ have the same inner normal fan $\Sigma$.
Furthermore, the coordinates of their vertices vary linearly with $\vec c$, i.e., there
are linear maps $\vec v_\tau: \Q^r \to \Q^n$, indexed by the set $\Sigma(n)$ of maximal cones $\tau$ of $\Sigma$, such that the vertices of $\Delta(\vec c)$ are given by the
$\vec v_\tau(\vec c)$. The vertex cone of $\Delta(\vec c)$ at $\vec v_\tau(\vec c)$ is then given by
$\vec v_\tau(\vec c) + K_\tau$ for the fixed cone $K_\tau := \tau^\vee$.
In particular, the contracted classes of
$\Delta(\vec c)$ for $\vec c \in \relint(C)$ do not depend on $\vec c$.
We call them the contracted classes of $\Delta$ and regard them as sets of maximal cones of $\Sigma$.
We say that $\Delta$ is a
\emph{lattice family} if $\vec v_\tau(\vec c) \in \Z^n$ for $\vec c \in \Z^r$.

By Brion's theorem in its form for indicator functions \cite[proof of Theorem 3.5]{Barvinok1999AnAT},
for every $\vec c \in \relint(C)$ we have, modulo lineality,
\[
    \mathbf 1_{\Delta(\vec c)} \equiv \sum_{\tau \in \Sigma(n)} \mathbf 1_{\vec v_\tau(\vec c) + K_\tau}.
\]
With the convex chain $\Delta(\vec c)$ on the left, this congruence extends to all $\vec c \in \Q^r$ by \cite[\S4, Proposition 2]{PukhlikovKhovanskii92}, so in particular, we get that
\begin{equation}\label{eq:degenerate-brion}
\varsigma(\Delta(\vec c); \vec x) = \sum_{\tau \in \Sigma(n)} \varsigma(\vec v_\tau(\vec c) + K_\tau; \vec x)
\end{equation}
holds for all $\vec c \in \Q ^r$. Note that $\Delta(\vec c)$ has bounded support for all $\vec c$, so that $\varsigma(\Delta(\vec c); \vec x)$ is a Laurent polynomial and $\varsigma_\pi(\Delta(\vec c); \vec y) := \pi_*\varsigma(\Delta(\vec c); \vec x)$ is always defined.

\begin{theorem}\label{thm:linear-family}
    Let $\Delta$ be a lattice family as above. For a contracted class $S$ of $\Delta$ let $E_S$ hold one primitive direction vector for each edge direction occurring at a vertex of $S$, taken up to sign. Let also $d_S$ be the
    dimension of the span of $E_S \cap \ker \pi$. Then there are linear maps
    $\vec a_S: \Z^r \to \Z^m$ and polynomials $g_S$ in $\vec c$ of degree at most $d_S$ with
    coefficients in $\K[\vec y^{\pm 1}]$
    such that for all $\vec c \in \Z^r$
    \[
        \varsigma_\pi(\Delta(\vec c); \vec y) = \sum_S \vec y^{\vec a_S(\vec c)}\,
        \frac{g_S(\vec c)}{\prod_{\vec u \in E_S,\, \pi(\vec u) \neq 0}\bigl(1 - \vec y^{\pi(\vec u)}\bigr)} .
    \]
\end{theorem}

\begin{proof}
    By \eqref{eq:degenerate-brion}, grouping the $\tau$ by contracted
    classes, for any $\vec c \in \Z^r$,
    \[
        \varsigma(\Delta(\vec c); \vec x) = \sum_S
        \underbrace{
            \sum_{\tau \in S} \vec x^{\vec v_\tau(\vec c)}\varsigma(K_\tau; \vec x)
        }_{=: \psi_S(\vec c)} .
    \]

    Fix a contracted class $S$. Its vertices lie in one
    fibre of $\pi$, so $\vec a_S(\vec c) := \pi(\vec v_\tau(\vec c))$ does not depend on
    $\tau \in S$.
    Triangulate each $K_\tau$ without adding new
    rays. Then $K_\tau$ is a signed sum of the simplicial cones $\nu$ appearing in the triangulation.
    For any such cone $\nu$,
    $\varsigma(\nu; \vec x) = p_\nu / \prod_{\vec u \in \gen(\nu)} (1 - \vec x^{\vec u})$ for a
    Laurent polynomial $p_\nu$, where $\gen(\nu)$ is the set of primitive ray generators of $\nu$.
    Each ray of $\nu$ appears in $E_S$ up to sign, so there exist Laurent polynomials $q_\nu$ such that
    \[
        \psi_S(\vec c) = \frac{1}{\prod_{\vec u \in E_S,\, \pi(\vec u) \neq 0}(1 - \vec x^{\vec u})}
        \sum_{\tau \in S} \sum_\nu
        \underbrace{
        \frac{\vec x^{\vec v_\tau(\vec c)} q_\nu}{\prod_{\vec u \in \gen(\nu),\, \pi(\vec u) = 0} (1 - \vec x^{\vec u})}}_{=: f_\nu(\vec c)} .
    \]

    Choose an integer vector $\vec t$ with $\langle \vec u, \vec t \rangle \neq 0$ for every $\vec u$ appearing in some $E_S$, and define the ring homomorphism
    \begin{align*}
        \phi: \K[\vec x^{\pm 1}] &\to \K[\vec y^{\pm 1}]\llbracket s \rrbracket\\
        \vec x^{\vec u} &\mapsto \vec y^{\pi(\vec u)} (1 + s)^{\langle \vec u, \vec t \rangle}.
    \end{align*}
    Then $\pi_*$ is the constant term of $\phi$ on $\K[\vec x^{\pm 1}]$.
    Note that 
    $\phi$ extends to a map $\K[\vec x^{\pm 1}]_{\ker \phi} \to \K(\vec y)(\!(s)\!)$ and that furthermore
    $\phi(\K[\vec x^{\pm 1}]_{\ker \pi_*}) \subseteq \K(\vec y)\llbracket s \rrbracket$.
    Since $\phi(1 - \vec x^{\vec u}) \neq 0$ for all $\vec u$ in the sets $E_S$,
    none of the denominators in our case lie in $\ker \phi$.
    Applying $\phi$ to $f_\nu(\vec c)$ takes each of the $n_\nu$ factors of the denominator to $s$
    times a unit of $\K\llbracket s \rrbracket$ and $\vec x^{\vec v_\tau(\vec c)}$ to
    $\vec y^{\vec a_S(\vec c)} (1 + s)^{\langle \vec v_\tau(\vec c), \vec t \rangle}$. So 
    \[
        \phi(f_\nu(\vec c)) = \vec y^{\vec a_S(\vec c)} s^{-n_\nu}(1 + s)^{\langle \vec v_\tau(\vec c), \vec t \rangle} \cdot \alpha_\nu,
    \]
    where $\alpha_\nu \in \K[\vec y^{\pm 1}]\llbracket s \rrbracket$ does not depend on $\vec c$.
    Then
    \[
        \phi(\psi_S(\vec c)) = \frac{\vec y^{\vec a_S(\vec c)}\, \widetilde{g}_S(\vec c)}
        {\prod_{\vec u \in E_S,\, \pi(\vec u) \neq 0} \phi(1 - \vec x^{\vec u})},
        \quad\text{where}\quad
        \widetilde{g}_S(\vec c) = \sum_{\tau \in S}\sum_{\nu} s^{-n_\nu}(1 + s)^{\langle\vec v_\tau(\vec c), \vec t\rangle}
        \cdot \alpha_\nu,
    \]
    and the factors of the denominator are units of $\K(\vec y)\llbracket s \rrbracket$ with
    constant terms $1 - \vec y^{\pi(\vec u)}$.
    The coefficient of $s^j$ in $(1 + s)^b$ is $\binom{b}{j}$, a polynomial of degree $j$
    in $b$, and $\langle \vec v_\tau(\vec c), \vec t \rangle$ is linear in $\vec c$, so for every
    $j \in \Z$ the coefficient of $s^j$ in
    $\widetilde{g}_S$ is a polynomial in
    $\vec c \in \Z^r$ with coefficients in $\K[\vec y^{\pm 1}]$. For all $j < 0$ it vanishes on
    $\relint(C) \cap \Z^r$. Indeed, for such $\vec c$, $\psi_S(\vec c) \in \K[\vec x^{\pm 1}]_{\ker \pi_*}$
    by \Cref{thm:projected-brion}, so $\phi(\psi_S(\vec c))$, and with it
    $\widetilde{g}_S(\vec c)$, lies in $\K(\vec y)\llbracket s \rrbracket$.
    Since $C$ is a full-dimensional cone, it follows that $[s^j]\widetilde{g}_S(\vec c) = 0$ and so
    $\widetilde{g}_S(\vec c) \in \K[\vec y^{\pm 1}]\llbracket s \rrbracket$
    for all $\vec c \in \Z^r$. Let 
    \[g_S(\vec c) = [s^0] \widetilde{g}_S(\vec c) = \sum_{\tau \in S} \sum_{\nu} [s^{n_\nu}] (1 + s)^{\langle \vec v_\tau(\vec c), \vec t \rangle} \alpha_\nu.\]
    Then $g_S$ is a polynomial in $\vec c$ with coefficients in $\K[\vec y^{\pm 1}]$ of degree at most the maximum of the $n_\nu$. The
    rays of $\nu$ are linearly independent, so $n_\nu \leq d_S$.

    Finally, we have that
    \[[s^0]\phi(\psi_S(\vec c)) = \frac{\vec y^{\vec a_S(\vec c)} g_S(\vec c)}{
    \prod_{\vec u \in E_S,\, \pi(\vec u) \neq 0} (1 - \vec y^{\pi(\vec u)})},\] 
    and so we conclude using \eqref{eq:degenerate-brion} by noting that
    $\varsigma_\pi(\Delta(\vec c); \vec y) = [s^0]\phi(\varsigma(\Delta(\vec c); \vec x))$.
\end{proof}

\begin{remark}\label{rem:rational-families}
    In our setup we assumed the existence of a full-dimensional cone $C$ such that $\Delta(C) \subseteq \cP(\Q^n)$. This is not strictly necessary. Write
    $\Delta(\vec e_i) = \mathbf 1_{P_i^+} * \mathbf 1_{P_i^-}^{-1}$ with rational polytopes $P_i^\pm$,
    and let $Q := \sum_i (P_i^+ + P_i^-)$. By linearity of $\Delta$,
    \[
        \Delta(\vec c) * \mathbf 1_{tQ} = \mathbf 1_{\sum_i (t + c_i) P_i^+ + (t - c_i) P_i^-},
    \]
    which is a polytope whenever $t \geq \lvert c_i \rvert$ for all $i$. So $\Delta$ is the slice
    $t = 0$ of the linear family $(\vec c, t) \mapsto \Delta(\vec c) * \mathbf 1_{tQ}$, which
    satisfies our assumption on the cone $\{t \geq \lvert c_i \rvert \text{ for all } i\}$, and which is
    a lattice family if the $P_i^\pm$ are lattice polytopes. 
    
    \Cref{thm:linear-family} also extends to families which are not lattice families.
    The vectors $\vec c \in \Z^r$ with $\vec v_\tau(\vec c) \in \Z^n$ for all $\tau$ form a sublattice
    $L$ of finite
    index, and on a coset $\vec c_0 + L$ the proof applies verbatim to the terms
    $\vec x^{\vec v_\tau(\vec c - \vec c_0)} \varsigma(\vec v_\tau(\vec c_0) + K_\tau; \vec x)$, so that
    the same formula holds with $\vec a_S$ affine and $g_S$ polynomial on each coset of $L$, 
    hence $g_S$ quasi-polynomial in $\vec c$.
\end{remark}

\begin{remark}\label{rem:peterson}
    Peterson \cite[Theorem 4.4]{Peterson23} obtained a similar result in a slightly different
    setting. He considered the sum
    \[
        S(P; \xi) = \sum_{\vec m \in \Z^n \cap P} e^{\langle \xi, \vec m\rangle}
    \]
    for a fixed linear functional $\xi\colon \Z^n \to \C$, and showed that
    \[
        S(\Delta(\vec c); \xi) = \sum_\tau Q_\tau(\vec c)\, e^{\langle \xi, \vec v_\tau(\vec c) \rangle}
    \]
    with polynomials $Q_\tau$ (see also \cite[(14)]{Peterson25}, where he notes that it follows from
    his degenerate Brion formula). In his setting $\xi$ plays the role of our projection $\pi$.
    The proof of \Cref{thm:linear-family} follows an argument similar to Peterson's, but the result is
    packaged differently. Here the exponents of $\vec y$ keep track of the fibres of $\pi$, and the
    coefficients of $\varsigma_\pi(\Delta(\vec c); \vec y)$ are the fibre counts. In his setting the
    fibres appear only through the exponents $\langle \xi, \vec m \rangle$, and recovering the counts
    requires varying $\xi$ (and with it the $Q_\tau$).
\end{remark}

\subsection{Gelfand--Zetlin and string polytopes}\label{sec:gz-string}

We end this section with examples of linear families coming from representation theory. First we recall the definition of Gelfand--Zetlin polytopes.

For a dominant weight $\lambda = (\lambda_1, \dots, \lambda_n)$ of $\mathrm{GL}_n$, that is,
$\lambda_1 \geq \dots \geq \lambda_n$, the Gelfand--Zetlin polytope $\mathrm{GZ}(\lambda)$ is the
set of real triangular arrays $(a_{i,j})_{1 \leq i \leq n - 1,\, 1 \leq j \leq n - i}$ satisfying
the interlacing conditions
\[
    a_{i-1,j} \geq a_{i,j} \geq a_{i-1,j+1}, \qquad a_{0,j} := \lambda_j .
\]
Its lattice points index an eigenbasis of the irreducible representation $V_\lambda$ with highest weight $\lambda$, with respect to the maximal torus of $\mathrm{GL}_n$. The basis vector corresponding to an array
$(a_{i,j})$ has the weight $\mu$ with $\mu_i = \sum_j a_{i-1,j} - \sum_j a_{i,j}$ for
$1 \leq i \leq n$. The assignment $(a_{i,j}) \mapsto \mu$ is an affine map $\pi$, and its fibre counts are the weight multiplicities of $V_\lambda$. All the results thus far also apply to affine $\pi$, so one can apply \Cref{thm:projected-brion} to $\mathrm{GZ}(\lambda)$ to compute the character of $V_\lambda$. It was shown in \cite[Proposition 2.3]{Makhlin16} that the projection $\pi$ contracts no edges of
$\mathrm{GZ}(\lambda)$, so the contracted classes are singletons. Moreover, \cite[Theorem 1.3]{Makhlin16} shows that for regular $\lambda$, the polytope $\mathrm{GZ}(\lambda)$ has exactly $n!$ simple vertices,
whose contributions are precisely the $n!$ summands of Weyl's character formula, and every other vertex contributes $0$ after applying $\pi_*$.
Finally, the map $\lambda \mapsto \mathrm{GZ}(\lambda)$ is additive on dominant weights, so these polytopes form a linear family.
This gives a way to define $\mathrm{GZ}(\lambda)$, and in particular its projected integer-point transform
$\sigma_\pi(\mathrm{GZ}(\lambda); \vec y)$, for an arbitrary weight $\lambda \in \Z^n$.

A similar construction exists for any connected reductive group $G$. Let $B \subseteq G$ be a Borel subgroup with maximal torus $T$, let $W$ be the Weyl group, and let
$\alpha_1, \dots, \alpha_\ell$ be the simple roots. Again, to each dominant weight $\lambda$ one associates a polytope $S(\lambda)$, the string polytope \cite{Littelmann98}. Unlike the Gelfand--Zetlin polytope, it depends on the additional choice of a reduced word $\underline w_0$ of the longest element $w_0$ of the Weyl group $W$. As in the case of $\mathrm{GL}_n$, the lattice points of the string polytope
$S(\lambda) \subseteq \R^N$ index a basis of the irreducible representation $V_\lambda$ of highest
weight $\lambda$, the point $\vec a \in S(\lambda)$ carrying the weight
$\lambda - \sum_k a_k \alpha_{i_k}$ \cite[Proposition 1.5]{Littelmann98}. Let
$\pi(\vec a) := \lambda - \sum_k a_k \alpha_{i_k}$ be the corresponding affine projection.
Then Littelmann's character formula reads
$\chi_T(V_\lambda) = \sigma_\pi(S(\lambda); \vec y)$ for every dominant
$\lambda$. In type $A$, for a suitable choice of reduced word, the string polytope $S(\lambda)$ recovers $\mathrm{GZ}(\lambda)$ up to a unimodular affine map depending on $\lambda$.

As in the case of Gelfand--Zetlin polytopes, string polytopes are homogeneous in $\lambda$, that is, $S(k\lambda) = kS(\lambda)$.
However, string polytopes are not additive in general: $S(\lambda+\mu) = S(\lambda) + S(\mu)$ can fail. Nevertheless, there is a fan $\mathcal{F}$ supported on the positive Weyl chamber such that over every cone $\tau \in \mathcal{F}$ the string polytopes form a linear family. For a maximal cone $\tau \in \mathcal{F}$, the restriction of the string polytopes to $\tau$ extends to a linear family $S_\tau$ on all weights, which allows us to define the projected integer-point transform $\sigma_\pi(S_\tau(\lambda); \vec y)$ for an arbitrary weight $\lambda$. One corollary of \Cref{thm:linear-family} (see \cite[Proposition 5.10]{MoninSmirnov26}) is that this extension of the projected integer-point transform does not depend on the choice of the maximal cone of $\mathcal{F}$; that is, for any two maximal cones $\tau, \tau' \in \mathcal{F}$ we have
\[
\sigma_\pi(S_\tau(\lambda); \vec y) = \sigma_\pi(S_{\tau'}(\lambda); \vec y) \quad \text{for every weight } \lambda.
\]
This observation was used in \cite{MoninSmirnov26} to describe the $T$-equivariant $K$-theory of generalized flag varieties
$G/B$.

We finish with an example of a string polytope whose weight map contracts edges.

\begin{example}[String polytopes]\label{ex:string-polytopes}
    Let $G = \mathrm{Sp}_6$ be the group of type $C_3$, and consider the string polytopes of
    Littelmann's reduced word, which in Bourbaki's labelling is $(3,2,3,2,1,2,3,2,1)$
    \cite[Section 6]{Littelmann98}. As usual, we denote by $\rho$ the sum of the fundamental weights. Then the points $(0,1,1,0,1,1,1,1,0)$ and $(0,2,2,0,1,1,0,0,0)$ span
    an edge of $S(\rho)$, and $\pi$ kills its direction since
    $\alpha_{i_2} + \alpha_{i_3} = \alpha_{i_7} + \alpha_{i_8}$. The transforms of the vertex cones
    at such an edge need not have an image under $\pi_*$, and only their sum over the contracted
    class does.
\end{example}

Although the contracted classes of string polytopes need not be singletons, we expect the results of
\cite{Makhlin16} to extend to string polytopes. More concretely, for every reduced word of $w_0$ and every dominant
$\lambda$, each vertex of $S(\lambda)$ whose weight $\mu$ lies in the orbit $W\lambda$ should form a
contracted class of its own, and the corresponding term in projected Brion's theorem should coincide with the sum of the summands of Weyl's character formula over
the $w \in W$ with $w\lambda = \mu$. Moreover, the term of every other contracted class should be $0$.
\Cref{thm:projected-brion} would then recover Weyl's character formula term by term for all string
polytopes. Beyond the Gelfand--Zetlin case, where it is Makhlin's theorem, we have checked this by computer for all $42$ reduced words of
$w_0$ in each of the types $B_3$ and $C_3$, at several regular and singular weights, evaluating the term of every
contracted class. We plan to address the general statement in future work.

\printbibliography

\Addresses

\end{document}